\documentclass{amsart}
\usepackage{amsmath,amsfonts,amsthm,amssymb}
\usepackage[square,comma]{natbib}
\def\multiset#1#2{\ensuremath{\left(\kern-.3em\left(\genfrac{}{}{0pt}{}{#1}{#2}\right)\kern-.3em\right)}}

\newcommand{\bbf}{\mathbb{F}}
\newcommand{\bbz}{\mathbb{Z}}
\newcommand{\bbr}{\mathbb{R}}
\newcommand{\bbc}{\mathbb{C}}
\newcommand{\A}{\mathfrak{A}}
\newcommand{\rk}{\mathrm{rk}}
\newcommand*{\bbe}{\mathbb{E}}
\newcommand{\abs}[1]{\left\lvert #1\right\rvert}
\newcommand{\brac}[1]{\left( #1\right)}
\newcommand{\norm}[1]{\left\lVert #1\right\rVert}
\newtheorem{theorem}{Theorem}[section]
\newtheorem{lemma}{Lemma}[section]
\newtheorem{corollary}{Corollary}[section]
\begin{document}
\title[A quantitative improvement for Roth's theorem]{A quantitative improvement for Roth's theorem on arithmetic progressions}
\author{Thomas F. Bloom}
\address{Thomas Bloom\\Department of Mathematics\\University of Bristol\\
University Walk\\Clifton\\ Bristol BS8 1TW\\United Kingdom}
\email{matfb@bristol.ac.uk}
\thanks{The author was supported by an EPSRC doctoral training grant.}
\begin{abstract}
We improve the quantitative estimate for Roth's theorem on three-term arithmetic progressions, showing that if $A\subset\{1,\ldots,N\}$ contains no non-trivial three-term arithmetic progressions then $\abs{A}\ll N(\log\log N)^4/\log N$. By the same method we also improve the bounds in the analogous problem over $\bbf_q[t]$ and for the problem of finding long arithmetic progressions in a sumset.
\end{abstract}	
\maketitle

\section{Introduction}
In this paper we prove the following quantitative improvement for Roth's theorem on arithmetic progressions.
\begin{theorem}\label{first}
If $A\subset\{1,\ldots,N\}$ contains no non-trivial three-term arithmetic progressions then
\[\abs{A}\ll \frac{(\log\log N)^4}{\log N}N.\]
\end{theorem}
This problem has a long history, the first significant quantitative bound being given by \cite{Ro:1953}. Let $R(N)$ denote the size of the largest subset of $\{1,\ldots,N\}$ which contains no non-trivial three-term arithmetic progressions. For comparison the table below summarises the history of upper bounds for $R(N)$.
\\
\\
\begin{tabular}[c]{|c|c|}
\hline
\cite{Ro:1953} & $N/\log\log N$\\
\hline
\cite{Sz:1990} and \cite{He:1987} & $N/(\log N)^c$ for some $c>0$\\
\hline
\cite{Bo:1999} & $(\log\log N)^{1/2}N/(\log N)^{1/2}$\\
\hline
\cite{Bo:2008} & $(\log\log N)^2N/(\log N)^{2/3}$\\
\hline
\cite{Sa:2012} & $N/(\log N)^{3/4-o(1)}$\\
\hline
\cite{Sa:2011} & $(\log\log N)^6N/\log N$\\
\hline
\end{tabular}
\\
\\
The claimed bound of $(\log\log N)^5N/\log N$ in \cite{Sa:2011} is due to a calculation error, and the method there in fact delivers $(\log\log N)^6N/\log N$ as in the table above. The best known lower bound has the shape $R(N)\gg N\exp(-c\sqrt{\log N})$ for some absolute constant $c>0$, due to \cite{Be:1946}, and so Theorem~\ref{first} still leaves much to be desired.

Not only does the method of this paper deliver a quantitative improvement but it manages to do so without the powerful combinatorial tools used in \cite{Sa:2011}, and instead operates almost entirely in `frequency space', exploiting a new lemma concerning structural properties of the large Fourier spectrum inspired by a recent breakthrough in Roth's theorem in $\bbf_p^n$ by \cite{BaKa:2012}. In particular, we hope that the methods in this paper could be used along with more combinatorial techniques to yield further quantitative progress.

The significant new ingredient in the proof of Theorem~\ref{first} is a new lemma about the structural properties of the set of large Fourier coefficients of a given set, which under certain circumstances offers a quantitative improvement over a related well-known lemma of \cite{Ch:2002}. As a further demonstration of the utility of this lemma we outline how it can be combined with the technique of \cite{Sa:2008} to prove the following quantitative improvement to the problem of finding long arithmetic progressions in a sumset.
\begin{theorem}\label{aps}
If $A,B\subset \{1,\ldots,N\}$ with both $c_1\alpha N\leq \abs{A} \leq \abs{B}\leq c_2\alpha N$ then $A+B$ contains an arithmetic progression of length at least
\[\exp\brac{cf(\alpha)\sqrt{\log N}},\]
where $f(\alpha)=\alpha^{1/2}/\log(1/\alpha)$ and the constant $c>0$ depends only on $c_1$ and $c_2$. 
\end{theorem}
By contrast, \cite{Sa:2008} used Chang's lemma to prove a similar result with $f(\alpha)=\alpha$, a result which was first proved using a different method by \cite{Gr:2002}. This was subsequently improved by \cite{CrLaSi:2013} to $f(\alpha)=\alpha^{1/2}\allowbreak(\log(1/\alpha))^{-3/2}$. The proof of Theorem~\ref{aps} can be easily adapted to give a similar bound when the sizes of $A$ and $B$ are not comparable but in this situation the method of \cite{CrLaSi:2013} delivers superior bounds.

We shall present our method in a general setting, which will allow us to prove more general versions of Theorem~\ref{first} in both the integers and their finite-characteristic analogue $\bbf_q[t]$. In general, we will be concerned with counting non-trivial solutions to equations of the shape
\begin{equation}\label{main}c_1x_1+c_2x_2+c_3x_3=0,\end{equation}
where $c_1+c_2+c_3=0$. A trivial solution is one where $x_1=x_2=x_3$. Applied to any such equation our method yields the following result over the integers.
\begin{theorem}\label{thm1}
If $c_1,c_2,c_3\in\bbz\backslash\{0\}$ are such that $c_1+c_2+c_3=0$ and $A\subset\{1,\ldots,N\}$ contains no non-trivial solutions to \eqref{main} then
\[\abs{A}\ll_{\mathbf{c}} \frac{(\log\log N)^4}{\log N}N.\]
\end{theorem}
In particular, Theorem~\ref{first} follows by considering the coefficients $\mathbf{c}=(1,1,-2)$. The polynomial ring $\bbf_q[t]$ is, in many respects, a finite-characteristic analogue of the integers, so we should expect a result of similar strength to hold. In \cite{Bl:2012} the method of \cite{Sa:2011} was adapted to give a quantitative result for $\bbf_q[t]$. In this paper we shall similarly prove an analogue of Theorem~\ref{thm1} for $\bbf_q[t]$. As in \cite{Bl:2012} the finite-characteristic property leads to a slight improvement in this case (due to the preponderance of subgroups).
\begin{theorem}\label{thm2}
Let $A\subset \bbf_q[t]_{\deg<n}$. If $c_1,c_2,c_3\in\bbf_q[t]\backslash\{0\}$ are such that $c_1+c_2+c_3=0$ and $A$ contains no non-trivial solutions to \eqref{main} then
\[\abs{A}\ll_{\mathbf{c}} \frac{(\log n)^2}{n}q^n.\]
\end{theorem}
We remark that if $c_1,c_2,c_3\in\bbf_q\backslash\{0\}$ then the problem is much simpler and better bounds are available. In particular, \cite{LiSp:2009} proved that in this case we have the upper bound $\abs{A}\ll_{\mathbf{c}} q^n/n$; when $c_1,c_2,c_3\in \bbf_p\backslash\{0\}$ this bound was first provided by \cite{Me:1995}. It is also likely that the work of \cite{BaKa:2012} in $\bbf_p^n$ could be adapted to deliver a bound of $\abs{A}\ll_{\mathbf{c}} q^n/n^{1+\epsilon}$ for some absolute constant $\epsilon>0$, when $c_1,c_2,c_3\in \bbf_q\backslash\{0\}$. 

The core of our argument is a qualitatively stronger alternative to a well-known lemma of \cite{Ch:2002} concerning the additive structure of the large spectrum of the Fourier transform in finite groups. More particularly, if $G$ is a finite abelian group and $A\subset G$ with density $\alpha=\abs{A}/\abs{G}$ then we define the large spectrum $\Delta_\eta(A)$ as the set of characters $\gamma\in\widehat{G}$ such that $\abs{\widehat{A}(\gamma)}\geq \eta\abs{A}$. It follows immediately from Parseval's theorem that $\abs{\Delta_\eta(A)}\leq \eta^{-2}\alpha^{-1}$, which is the best possible bound on the cardinality. Chang's lemma shows that a smaller set can be found which additively `controls' the spectrum. In particular, we say that $\Delta$ is $d$-covered if there exists $\Lambda$ of size $\abs{\Lambda}\leq d$ such that
\[\Delta\subset \left\{ \sum_{\lambda\in\Lambda}\epsilon_\lambda\lambda : \epsilon_\lambda\in \{-1,0,1\}\right\}.\]
In this language, Chang's lemma may be stated as follows.
\begin{theorem}[\cite{Ch:2002}]
If $A\subset G$ with density $\alpha=\abs{A}/\abs{G}$ then $\Delta_\eta(A)$ is $d$-covered for some 
\[d\ll \eta^{-2}\log(1/\alpha).\]
\end{theorem}
This important structural lemma has found many applications since, and in particular has been instrumental in the recent advances in Roth's theorem. By contrast, in this paper our method does not directly use Chang's lemma, but rather the following theorem which is, for some applications, much stronger.

\begin{theorem}\label{thch}
If $A\subset G$ with density $\alpha=\abs{A}/\abs{G}$ then there exists $\Delta'\subset\Delta_\eta(A)$ of size $\abs{\Delta'}\gg \eta\abs{\Delta_\eta(A)}$ which is $d$-covered for some 
\[d\ll \eta^{-1}\log(1/\alpha).\]
\end{theorem}
In particular, we can save a factor of $\eta$ in the dimension while only losing a factor of $\eta$ on the size of the set considered. We prove Theorem~\ref{thch}, or rather, a more general version, by considering the additive energy of large spectra, building on work by \cite{Sh:2008} and \cite{BaKa:2012}.

For our applications we shall discuss the ideas leading to Theorem~\ref{thch} in a general setting that also applies to covering `relative' to a given set, which is needed for the density increment strategy used for Roth's theorem.

\section{Notation and definitions}
We fix some finite abelian group $G$ with dual group $\widehat{G}$. Let $N=\abs{G}$. For convenience we shall use the counting measure on both $G$ and $\widehat{G}$. In particular all $L^p$ norms on both $G$ and $\widehat{G}$ are defined with respect to the counting measure, so that if $f:G\to\bbc$ and $p\geq 1$ then
\[\norm{f}_p=\brac{\sum_x \abs{f(x)}^p}^{1/p}\textrm{ and }\norm{f}_\infty=\sup_{x\in G}\abs{f(x)},\]
and similarly for $\omega:\widehat{G}\to\bbc$. For any function $f:G\to\bbc$ we define the Fourier transform $\widehat{f}:\widehat{G}\to\bbc$ by
\[\widehat{f}(\gamma)=\sum_x f(x)\gamma(x).\]
For all functions $f,g:G\to\bbc$ we have Parseval's identity
\[\langle f,g\rangle=\sum_x f(x)\overline{g(x)}=N^{-1}\sum_\gamma \widehat{f}(\gamma)\overline{\widehat{g}(\gamma)}.\]
If $B\subset G$ and $\Gamma\subset\widehat{G}$ then for any $\epsilon\in[0,2]$ we say that $B$ has $\epsilon$-control of $\Gamma$ if for all $x\in B$ and $\gamma\in\Gamma$ we have
\[\abs{1-\gamma(x)}\leq \epsilon.\]
This condition will be important in several places; in general, the hypothesis of control presents no serious difficulties as we will just define the sets we are working with precisely so that they have the required control. The details will vary on the choice of the group $G$, and hence for most of the paper we simply present the necessary control hypothesis and take care of how to ensure it in our applications in the final sections of the paper. 

For any $\eta\in[0,1]$ and $f:G\to\bbc$ we define the spectrum 
\[\Delta_\eta(f)=\left\{ \gamma\in\widehat{G} : \abs{\widehat{f}(\gamma)}\geq \eta\norm{f}_1\right\}\]
and the level spectrum
\[\tilde{\Delta}_\eta(f)=\left\{ \gamma\in\widehat{G} : \eta\norm{f}_1\leq \abs{\widehat{f}(\gamma)}
<2\eta\norm{f}_1\right\}.\]
For $0<\delta\leq 1$ we shall use the convenient shorthand $\mathcal{L}(\delta)$ to denote $2+\lceil \log(1/\delta)\rceil$. Finally, it will often be convenient for certain sets to renormalise the counting measure to be compact; these sets will always be denoted by $B$ (possibly with some subscripts or superscripts) and then for any $A\subset G$ we define $\beta(A)=\abs{A\cap B}/\abs{B}$. In general, if we speak of $A\subset B$ having relative density $\alpha$ then this means $\beta(A)=\alpha$. 

We will frequently abuse notation by conflating a set and its characteristic function; thus, for example, if $\Gamma\subset\widehat{G}$ then $\Gamma(\gamma)=1$ if $\gamma\in\Gamma$ and $0$ otherwise.

For any $\Gamma\subset\widehat{G}$ and $\omega:\widehat{G}\to\bbr_+$ and integer $m\geq 1$ we define the additive energy as
\[E_{2m}(\omega,\Gamma)=
\sum_{\gamma_1,\ldots,\gamma_{m}'}\omega(\gamma_1)\cdots\omega(\gamma_{m}')
\Gamma\brac{\sum_{i=1}^{m}\gamma_i-\sum_{j=1}^{m}\gamma_{j}'}.\]
Similarly, we define the restricted energy as 
\[E_{t_1,t_2}^\sharp(\omega,\Gamma)=
\sum_{\substack{\Delta_1\in\binom{\widehat{G}}{t_1},\Delta_2\in\binom{\widehat{G}}{t_2}\\ \Delta_1\cap\Delta_2=\emptyset}}\prod_{\gamma\in\Delta_1\cup\Delta_2}\omega(\gamma)
\Gamma\brac{\sum_{\gamma\in\Delta_1}\gamma-\sum_{\gamma'\in\Delta_2}\gamma'}.\]
We write $E_{2m}^\sharp$ for $E_{m,m}^\sharp$ and for any $\omega$ and $\Gamma$ we define $E_0(\omega,\Gamma)=E_0^\sharp(\omega,\Gamma)=1$. Observe that $E$ and $E^\sharp$ differ, not only in the restriction on repeating elements, but also in that the former is sensitive to permutations of the $\gamma_i$. We say that $S$ is $d$-covered by $\Gamma$ if there exists $\Lambda\subset\widehat{G}$ of size $\abs{\Lambda}\leq d$ such that
\[S\subset \Gamma-\Gamma+\langle \Lambda\rangle,\]
where
\[\langle \Lambda\rangle = \left\{ \sum_{\lambda\in\Lambda}\epsilon_\lambda\lambda : \epsilon_\lambda\in \{-1,0,1\}\right\}\]
and $\langle \emptyset\rangle=\{0\}$.

We say that $\Delta$ is $\Gamma$-dissociated if for all $k\geq 1$ and $\lambda\in\widehat{G}$ there are at most $2^k$ many pairs $\Delta_1,\Delta_2$ of disjoint subsets of $\Delta$ such that $\abs{\Delta_1\cup\Delta_2}=k$ and
\[\sum_{\gamma\in\Delta_1}\gamma-\sum_{\gamma'\in\Delta_2}\gamma'\in \Gamma+\lambda.\]
Finally, we say that $S$ has $\Gamma$-dimension of $d$ if $d$ is the size of the largest $\Gamma$-dissociated subset of $S$. We observe that the dimension is always at least 1 since any singleton set is trivially $\Gamma$-dissociated for all $\Gamma\subset\widehat{G}$. Furthermore, if $\Delta$ is $\Gamma$-dissociated then it is also $\Gamma'$-dissociated for any translate $\Gamma'$ of $\Gamma$. 

\section{Additive energy}
In this section we discuss the relationship between the dimension of a set and its additive energy. If a set has a very large dimension then almost all of the set is dissociated, and hence one would expect few additive relations between its elements, so it should have small additive energy. The following lemma verifies this intuition.
\begin{lemma}\label{en1}
If $S\subset\widehat{G}$ has $\Gamma$-dimension $\abs{S}-k$ then for all $m\geq t_1,t_2\geq 0$
\[E^\sharp_{t_1,t_2}(S,\Gamma)\leq 4^{k+m}.\]
\end{lemma}
\begin{proof}
Let $S=S_0\sqcup S_1$ where $S_0$ is $\Gamma$-dissociated and $\abs{S_1}=k$. By separating the contribution from the subsets of $S_1$ we obtain the estimate 
\begin{align*}
E^\sharp_{t_1,t_2}(S,\Gamma)
&=\sum_{\substack{\Delta_1\in\binom{S}{t_1},\Delta_2\in\binom{S}{t_2}\\ \Delta_1\cap\Delta_2=\emptyset}}
\Gamma\brac{\sum_{\gamma\in\Delta_1}\gamma-\sum_{\gamma'\in\Delta_2}\gamma'}\\
&\leq \sum_{\substack{0\leq r_1\leq t_1\\ 0\leq r_2\leq t_2}}
\binom{k}{r_1}\binom{k}{r_2}
\sup_{\lambda}\sum_{\substack{\Delta_1\in\binom{S_0}{t_1-r_1},\Delta_2\in\binom{S_0}{t_2-r_2}\\ \Delta_1\cap\Delta_2=\emptyset}}
\Gamma\brac{\sum_{\gamma\in\Delta_1}\gamma-\sum_{\gamma'\in\Delta_2}\gamma'+\lambda}.
\end{align*}
Since $S_0$ is $\Gamma$-dissociated, however, the inner summand is bounded above by $2^{t_1+t_2}$ and the lemma follows.
\end{proof}
For the main result of this section we need to convert a conclusion about the restricted additive energy to the full additive energy, for which the following lemma will suffice.
\begin{lemma}\label{en2}
For any $\Gamma\subset \widehat{G}$, weight function $\omega:\widehat{G}\to\bbr_+$ and integer $m\geq 2$ we have
\[E_{2m}(\omega,\Gamma)\leq 2^{4m}(m!)^2\norm{\omega}_2^{2m}\sum_{0\leq t_1,t_2\leq m}
\frac{\norm{\omega}_2^{-t_1-t_2}}{((m-t_1)!(m-t_2)!)^{1/2}}\sup_\lambda E^\sharp_{t_1,t_2}(\omega,\Gamma+\lambda).\]
\end{lemma}
\begin{proof}
We divide the range of summation of $E_{2m}$ according to the size of the subsets of $\{\gamma_1,\ldots,\gamma_m\}$ and $\{\gamma_1',\ldots,\gamma_m'\}$ consisting of elements that each occur with multiplicity 1. This leads to the upper bound
\begin{equation}\label{eqone}E_{2m}(\omega,\Gamma)\leq \sum_{0\leq l_1,l_2\leq m}G_{m-l_1}(\omega)G_{m-l_2}(\omega)\binom{m}{l_1}\binom{m}{l_2}
\sup_\lambda F_\lambda(l_1,l_2)\end{equation}
where
\[F_\lambda(l_1,l_2)= \sum_{\substack{\gamma_1,\ldots,\gamma_{l_2}'\\ \gamma_i\neq \gamma_j\,\gamma_i'\neq\gamma_j'\,i\neq j}}\omega(\gamma_1)\cdots\omega(\gamma_{l_2}')
\Gamma\brac{\sum_{i=1}^{l_1}\gamma_i-\sum_{j=1}^{l_2}\gamma_{j}'-\lambda}\]
and
\begin{equation}\label{eqtwo}G_k(\omega)=\sum_{\Delta}^*\prod_{\gamma\in\Delta}\omega(\gamma)
\leq \frac{k!}{(\lfloor k/2\rfloor )!}\brac{\sum_{\gamma}\omega(\gamma)^2}^{k/2},\end{equation}
the first sum being restricted to those ordered $k$-tuples $\Delta\in \widehat{G}^k$ where each element occurs with multiplicity at least 2. The sum $F_\lambda(l_1,l_2)$ is almost a renormalised version of the restricted energy $E^\sharp_{l_1,l_2}$ except that it lacks the restriction $\gamma_i\neq \gamma_j'$ for all $1\leq i\leq l_1$ and $1\leq j\leq l_2$. To introduce this we partition $F_\lambda(l_1,l_2)$ according to the number of common elements between the $\gamma_i$ and $\gamma_i'$; thus
\begin{equation}\label{eqthree}F_\lambda(l_1,l_2)\leq \sum_{i=0}^{\min(l_1,l_2)}\binom{l_1}{i}\binom{l_2}{i}i!\norm{\omega}_2^{2i}(l_1-i)!(l_2-i)!E^\sharp_{l_1-i,l_2-i}(\omega,\Gamma+\lambda).\end{equation}
Combining \eqref{eqone}, \eqref{eqtwo} and \eqref{eqthree} and simplifying the expression implies that $E_{2m}(\omega,\Gamma)$ is at most
\[(m!)^2\norm{\omega}_2^{2m}\sum_{0\leq l_1,l_2\leq m}\sum_{i=0}^{\min(l_1,l_2)}\frac{\norm{\omega}_2^{2i-l_1-l_2}}{ i!(\lfloor (m-l_1)/2\rfloor)!(\lfloor (m-l_2)/2\rfloor)!}\sup_\lambda E_{l_1-i,l_2-i}^\sharp(\omega,\Gamma+\lambda).\]
Relabelling $t_1=l_1-i$ and $t_2=l_2-i$ this is at most
\[(m!)^2\norm{\omega}_2^{2m}\sum_{0\leq t_1,t_2\leq m}\norm{\omega}_2^{-t_1-t_2}\sup_\lambda E_{t_1,t_2}^\sharp(\omega,\Gamma+\lambda)f(m,t_1,t_2)\]
where
\[f(m,t_1,t_2)=\sum_{i\geq \max(t_1,t_2)}^{m-\max(t_1,t_2)}\frac{1}{i!(\lfloor (m-t_1-i)/2\rfloor)!(\lfloor (m-t_2-i)/2\rfloor)!}.\]
Finally, a tedious calculation using the elementary inequality $n!/(\lfloor n/2\rfloor!)^2\leq 2(n+1)^{1/2}2^{n}$, valid for all $n\geq 0$, shows that the inner sum is at most $2^{4m}((m-t_1)!(m-t_2)!)^{-1/2}$ and the lemma follows. 
\end{proof}

The final technical lemma of this section provides a relationship between covering and dimension that will be important in the proof of Theorem~\ref{en3}.

\begin{lemma}\label{techlemma}
Let $\Gamma\subset\widehat{G}$ be a symmetric set. If $\Delta\subset \widehat{G}$ has $\Gamma$-dimension $r$ then there is a partition $\widehat{G}=\Lambda_0\sqcup \Lambda_1$ where $\Lambda_0$ is $2r$-covered by $\Gamma$ and for all $\gamma\in\Lambda_1$ the set $\Delta\cup\{\gamma\}$ has $\Gamma$-dimension $r+1$.  
\end{lemma}
\begin{proof}
By hypothesis we can decompose $\Delta$ as $\Delta_0\sqcup\Delta_1$ where $\Delta_0$ is $\Gamma$-dissociated and $\abs{\Delta_0}=r$. Let $\Delta'$ be the set of all $\gamma\in\widehat{G}$ such that $\Delta_0\cup\{\gamma\}$ is not $\Gamma$-dissociated. We claim that if we let $\Lambda_0=\Delta'\cup\Delta_0$ and $\Lambda_1=\widehat{G}\backslash\Lambda_0$ then this is a suitable decomposition. 

Firstly, let $\gamma\in\Lambda_1$. By construction the set $\Delta_0\cup\{\gamma\}$ is $\Gamma$-dissociated, and hence $\Delta\cup\{\gamma\}$ has $\Gamma$-dimension at least $r+1$ by definition, since $\abs{\Delta_0\cup\{\gamma\}}=r+1$. It remains to show that $\Lambda_0$ is $2r$-covered by $\Gamma$; for this, it suffices to show that
\[\Delta_0\cup \Delta'\subset \Gamma-\Gamma+\langle \Delta_0\rangle +\langle \Delta_0\rangle.\] 
This is obvious for $\Delta_0$. Let $\gamma\in\Delta'$. By construction $\Delta_0\cup\{\gamma\}$ is not $\Gamma$-dissociated, and hence there exists $k\geq 1$ and $\lambda\in\widehat{G}$ such that there are more than $2^k$ many triples $(\epsilon,\Delta_1',\Delta_2')$ such that $\epsilon\in\{-1,0,1\}$, the sets $\Delta_1'$ and $\Delta_2'$ are disjoint subsets of $\Delta_0$ with $\abs{\Delta_1'\cup\Delta_2'}+\abs{\epsilon}=k$, and
\[\epsilon \gamma + \sum_{\gamma_1'\in\Delta_1'}\gamma_1'-\sum_{\gamma_2'\in\Delta_2'}\gamma_2'\in \Gamma+\lambda.\]
If there exists at least one such triple with $\epsilon=0$ and at least one with $\epsilon\neq0$ then it is easy to check that this implies that $\gamma\in \Gamma-\Gamma+\langle \Delta_0\rangle-\langle \Delta_0\rangle$ as required. If $\epsilon\equiv 0$ for all such triples then this contradicts the $\Gamma$-dissociativity of $\Delta_0$. Hence we can assume that $\epsilon\in\{-1,1\}$ for all such triples; this is clearly impossible for $k=1$, and for $k>1$ we observe that by the pigeonhole principle there are strictly more than $2^{k-1}$ such triples with identical $\epsilon$. This, however, is another contradiction to the $\Gamma$-dissociativity of $\Delta_0$, considering the translate $\Gamma+\lambda-\epsilon\gamma$. Thus $\gamma\in \Gamma-\Gamma+\langle \Delta_0\rangle-\langle \Delta_0\rangle$ as required, and the proof is complete.
\end{proof}

The following theorem is crucial, and uses random sampling to prove a hereditary version of our earlier intuition: namely, if a set is such that every large subset is not efficiently covered then we must have particularly small additive energy. The argument is a variant on that used in \cite[Section 5]{BaKa:2012}. There, however, they only wish to bound the $8$-fold additive energy, whereas for our purposes we shall need to deal with the $2m$-fold additive energy where $m\to\infty$ as $N\to\infty$ (for our applications we shall in fact take $m\approx (\log \log N)^{1+o(1)}$), and hence we have taken care to make the dependence on $m$ explicit.

We treat the constants in this argument, as in the rest of this paper, quite crudely; it is certainly possible to improve them, but such improvements would have a negligible effect on the main results.
\begin{theorem}\label{en3}
Let $\Gamma\subset\widehat{G}$ be a symmetric set and $\omega:\widehat{G}\to\bbr_+$. Let $m\geq 2$ and $d\geq n\geq 2$ be such that $m\leq d/4$ and $\norm{\omega}_2\leq m^{1/2}d^{-1}\norm{\omega}_1$. Then either there is a finite set $\Delta\subset\widehat{G}$ such that
\[\sum_{\gamma\in\Delta}\omega(\gamma)\geq
\frac{n}{d}\norm{\omega}_1\]
and $\Delta$ is $2d$-covered by $\Gamma$, or
\[E_{2m}(\omega,\Gamma)\leq 2^{13m+6n}m^{2m}d^{-2m}\norm{\omega}_1^{2m}.\]
\end{theorem}
\begin{proof}
Without loss of generality we may suppose that $\norm{\omega}_1=1$. We first observe that either we are in the first case, or every subset $\Delta\subset \widehat{G}$ which is $2d$-covered by $\Gamma$ satisfies $\sum_{\gamma\in\Delta}\omega(\gamma)\leq nd^{-1}$, which we shall assume henceforth.

Let $S\subset\widehat{G}$ be a random set of size at most $d$ chosen by selecting $d$ elements of $\widehat{G}$ at random, where we choose $\gamma\in\widehat{G}$ with probability $\omega(\gamma)$. We claim that for $k\geq 0$ the set $S$ has $\Gamma$-dimension $d-k$ with probability at most $n^k/k!$.

For suppose we have selected $d'\leq d$ elements of $S$, say $S'$, and suppose that $S'$ has $\Gamma$-dimension $r$. By Lemma~\ref{techlemma} we can partition $\widehat{G}=\Lambda_0\sqcup\Lambda_1$ such that $\Lambda_0$ is $2d$-covered by $\Gamma$ and for all $\gamma\in\Lambda_1$ the set $S'\cup\{\gamma\}$ has $\Gamma$-dimension at least $r+1$. Thus, in our model, 
\[\mathbb{P}(\dim(S'\cup\{\gamma\})\leq \dim(S'))\leq \sum_{\gamma\in\Lambda_0}\omega(\gamma)\leq \frac{n}{d},\]
since $\Lambda_0$ is $2d$-covered by $\Gamma$. From this estimate, combined with the trivial observations that the empty set has $\Gamma$-dimension $0$ and that $\Gamma$-dimension is non-decreasing, it follows that the probability that $S$ has $\Gamma$-dimension $d-k$ is at most the probability that $k$ events with probability at most $n/d$ occur in $d$ independent trials, which is at most 
\[\binom{d}{k}n^kd^{-k}\leq n^k/k!	\]
as required. By Lemma~\ref{en1} it follows that for all $\lambda\in\widehat{G}$ and integers $t_1,t_2\leq m$ we have
\[\bbe E_{t_1,t_2}^\sharp(S,\Gamma+\lambda)\leq 4^m\sum_{k=0}^\infty\frac{(4n)^k}{k!}= 4^me^{4n}.\]
Let $1\leq k\leq 2m$. For any distinct $\gamma_1,\ldots,\gamma_k\in \widehat{G}$ the probability that $\gamma_1,\ldots,\gamma_k\in S$ is at least
\[k!\binom{d}{k}\omega(\gamma_1)\cdots\omega(\gamma_k)\brac{1-\sum_{i=1}^k\omega(\gamma_i)}^{d-k}.\]
Since $k\leq d/2$ we have $k!\binom{d}{k}\geq (d/2)^k$. Furthermore, by the Cauchy-Schwarz inequality 
\[\sum_{i=1}^k\omega(\gamma_i)\leq (2m)^{1/2}\norm{\omega}_2\leq 2md^{-1}\leq 1/2,\]
so that the second factor is at least
\[\exp\brac{-d\sum_{i=1}^k\omega(\gamma_i)}\geq e^{-2m}.\]
It follows that the probability that $\gamma_1,\ldots,\gamma_k\in S$ is at least $2^{-5m}d^{k}\omega(\gamma_1)\cdots\omega(\gamma_k)$. By linearity of expectation, assuming $t_1+t_2\leq m$,
\[\bbe E_{t_1,t_2}^\sharp(S,\Gamma+\lambda)\geq 2^{-5m}d^{t_1+t_2}E_{t_1,t_2}^\sharp(\omega,\Gamma+\lambda),\]
and so, for all $\lambda\in\widehat{G}$ and $0\leq t_1,t_2\leq m$,
\[E^\sharp_{t_1,t_2}(\omega,\Gamma+\lambda)\leq 2^{7m}e^{4n}d^{-t_1-t_2}.\]
From Lemma~\ref{en2} it follows that
\[E_{2m}(\omega,\Gamma)\leq 2^{11m}e^{4n}m!\norm{\omega}_2^{2m}\brac{\sum_{0\leq t\leq m} \frac{(m!)^{1/2}(\norm{\omega}_2^{-1}d^{-1})^t}{((m-t)!)^{1/2}}}^2.\]
For brevity let $r=\norm{\omega}_2^{-2}d^{-2}\geq m^{-1}$. By the Cauchy-Schwarz inequality the inner factor is at most
\[(m+1)\sum_{0\leq t\leq m} \frac{m!}{(m-t)!}r^t\leq (m+1)\sum_{0\leq t\leq m}(erm)^t\leq (m+1)^2(erm)^m,\]
say. In particular
\[E_{2m}(\omega,\Gamma)\leq 2^{13m}e^{4n}m^{2m}\norm{\omega}_2^{2m}r^m=2^{13m}e^{4n}m^{2m}d^{-2m},\]
and the lemma follows.
\end{proof}

\section{Structure in Spectra}
The results in the previous section are extremely general, and can be used to deduce facts about the dimensions of arbitrary sets from lower bounds on their additive energy. We shall use them to derive a structural result about the large spectrum of a function $f:G\to\bbc$, and for this we require a lower bound on the additive energy of such spectra. A suitable lower bound was provided by \cite{Sh:2008}, who showed that if $A\subset G$ with density $\alpha=\abs{A}/N$ and $\Delta\subset\Delta_\eta(A)$ then $E_{2m}(\Delta,\{0\})\gg \eta^{2m}\alpha\abs{\Delta}^{2m}$. A simpler proof of this fact was given in \cite{Sh:2009}; the following is a simple generalisation of this proof. 

\begin{lemma}\label{sp}
Let $\epsilon\in[0,1]$. For any $B\subset G$ and $\eta\in[0,1]$ let $f:B\to\bbc$ and $\omega:\widehat{G}\to\bbr_+$ be supported on $\widehat{G}$. Then for all integers $m\geq 1$
\[E_{2m}(\omega,\Delta_\epsilon(B))\geq \norm{\omega}_1^{2m}\brac{
\brac{\eta\frac{\norm{f}_1}{\norm{f}_{2m/(2m-1)}\abs{B}^{1/2m}}}^{2m}-\epsilon}.\]
\end{lemma}
To recover the bound of Shkredov we set $B=G$ and let $\epsilon\to0$. 
\begin{proof}
Let $\chi$ be defined by letting $\widehat{\chi}(\gamma)=\overline{c_\gamma}\omega(\gamma)$, where $c_\gamma\widehat{f}(\gamma)=\abs{\widehat{f}(\gamma)}$. By construction whenever $\omega(\gamma)\neq 0$ we have $\abs{\widehat{f}(\gamma)}\geq \eta\norm{f}_1$, whence
\[\sum_xf(x)\overline{\chi(x)}=N^{-1}\sum_\gamma \widehat{f}(\gamma)\overline{\widehat{\chi}(\gamma)}=N^{-1}\sum_\gamma\omega(\gamma)\abs{\widehat{f}(\gamma)}\geq N^{-1}\eta\norm{f}_1\norm{\omega}_1.\]
By H\"{o}lder's inequality, however,
\[\brac{\sum_x f(x)\overline{\chi(x)}}^{2m}
\leq \brac{\sum_x \abs{f(x)}^{2m/(2m-1)}}^{2m-1}\brac{\sum_x B(x)\abs{\chi(x)}^{2m}}.\]
It remains to note that, by the triangle inequality,
\begin{align*}
\sum_x B(x)\abs{\chi(x)}^{2m}
&=N^{-2m}\sum_xB(x)\abs{\sum_\gamma c_\gamma\omega(\gamma)\gamma(x)}^{2m}\\
&\leq
N^{-2m}\sum_{\gamma_1,\ldots,\gamma_m'}\omega(\gamma_1)\cdots\omega(\gamma_m')
\abs{\widehat{B}(\gamma_1+\cdots+\gamma_m
-\gamma_1'-\cdots-\gamma_m')}.
\end{align*}
It follows that
\begin{align*}
\eta^{2m}\norm{f}_1^{2m}\norm{\omega}_1^{2m}
&\leq 
\norm{f}_{2m/(2m-1)}^{2m}\sum_{\gamma_1,\ldots,\gamma_m'}\omega(\gamma_1)\cdots\omega(\gamma_m')
\abs{\widehat{B}(\gamma_1+\cdots-\gamma_m')}\\
&\leq 
\norm{f}_{2m/(2m-1)}^{2m}\brac{\abs{B}E_{2m}(\omega,\Delta_\epsilon(B))
+\epsilon\abs{B}\norm{\omega}_1^{2m}},
\end{align*}
and the proof is complete.
\end{proof}

Finally, we can prove the technical heart of our argument, the aforementioned alternative to Chang's lemma. Again, for our application we need a fairly general statement, but at first glance the reader should take $B=G$ and any $\epsilon>0$, so that $\Delta_\epsilon(B)=\{0\}$. 
\begin{theorem}\label{then}
Suppose that $f:B\to \bbc$ and let $\alpha=\norm{f}_1/\norm{f}_\infty\abs{B}$. Let $\omega:\widehat{G}\to\bbr_+$ be supported on $\Delta_\eta(f)$ and let $0\leq \epsilon\leq \exp(-8\mathcal{L}(\eta)\mathcal{L}(\alpha))$. There is a set $\Delta'\subset\Delta_\eta(f)$ such that
\[\sum_{\gamma\in\Delta'}\omega(\gamma)\geq 2^{-12}\eta\norm{\omega}_1\]
and $\Delta'$ is $2^{14}\mathcal{L}(\alpha)\eta^{-1}$-covered by $\Delta_\epsilon(B)$.
\end{theorem}
\begin{proof}
Without loss of generality we may suppose that $\norm{\omega}_1=1$. Suppose first that $\norm{\omega}_2\geq 2^{-12}\mathcal{L}(\alpha)^{-1/2}\eta$ and let $\Delta'$ be a random set selected by including $\gamma\in\widehat{G}$ independently with probability $2^{13}\eta^{-1}\mathcal{L}(\alpha)\omega(\gamma)$. Then, if $\Delta'$ is this randomly chosen set we have, by Chernoff's inequality, that $\abs{\Delta'}\leq 2^{14}\eta^{-1}\mathcal{L}(\alpha)$ with probability at least $7/8$, say, and 
\[\bbe \sum_{\gamma\in\Delta'}\omega(\gamma)\geq 2^{13}\eta^{-1}\mathcal{L}(\alpha)\norm{\omega}_2^2\geq 2^{-11}\eta,\]
and hence by Markov's inequality we have $\sum_{\gamma\in\Delta'}\omega(\gamma)\geq 2^{-12}\eta$ with probability at least $1/2$, and the lemma follows. 

Otherwise, we let $n=m=\mathcal{L}(\alpha)$ and $d=\lfloor 2^{12}\eta^{-1}m\rfloor$ and apply Lemmata~\ref{sp} and \ref{en3}. By the above we can suppose that $\norm{\omega}_2\leq 2^{-12}\mathcal{L}(\alpha)^{-1/2}\eta\leq m^{1/2}d^{-1}$, as is necessary for the application of Lemma~\ref{en3}. Lemma~\ref{sp} implies that
\[E_{2m}(\omega;\Delta_\epsilon(B))\geq \brac{\eta\frac{\norm{f}_1}{\norm{f}_{2m/(2m-1)}\abs{B}^{1/2m}}}^{2m}-\epsilon.\]
We have the trivial bound $\norm{f}_{2m/(2m-1)}\leq \norm{f}_\infty^{1/2m}\norm{f}_1^{1-1/2m}$, and hence if $\epsilon\leq \eta^{2m}\alpha/2$ then 
\[E_{2m}(\omega;\Delta_\epsilon(B))\geq \eta^{2m}\alpha/2.\]
By Lemma~\ref{en3} either there is a set $\Delta'$ such that
\[\sum_{\gamma\in\Delta'}\omega(\gamma)\geq \frac{m}{d}\]
and $\Delta'$ is $2d$-covered by $\Delta_\epsilon(B)$, or
\[\eta^{2m}\alpha\leq 2^{19m+1}m^{2m}d^{-2m}.\]
In particular, 
\[d\leq 2^{10}m\eta^{-1}\alpha^{-1/2m},\]
which contradicts our initial choice of $d$ and $m$, and the proof is complete.
\end{proof}

Theorem~\ref{thch} is a special case of this; to obtain it one simply sets $B=G$, lets $\epsilon\to 0$ and takes $\omega$ as the characteristic function of $\Delta_\eta(f)$.

\section{The density increment}
In this section we show how to use the structural result Theorem~\ref{then} to obtain the usual density increment lemma which can be iterated to yield the main theorems. The density increment strategy originated with \cite{Ro:1953} but has seen various technical simplifications since which we incorporate here. Roughly speaking, the idea is to show that if $A$ does not contain the expected number of solutions to \eqref{main} then it has large Fourier coefficients, and this information can be translated into finding some group-like $B\subset G$ such that $\abs{A\cap B}/\abs{B}$ is larger than the expected $\abs{A}/N$, and the argument can then be iterated until it halts due to the trivial bound $\abs{A\cap B}\leq \abs{B}$.

The first lemma is the standard conversion of $L^2$-information on the balanced function of a set to a density increment, a technique first exploited for Roth's theorem by \cite{He:1987} and \cite{Sz:1990}.

\begin{lemma}\label{de1}
Let $f:B\to [0,1]$ and $\mathbf{f}=f-\alpha B$, where $\alpha=\norm{f}_1/\abs{B}$. Suppose that 
\[\sum_{\gamma\in\Gamma}\abs{\widehat{\mathbf{f}}(\gamma)}^2\geq \nu\alpha\norm{f}_1 N.\]
Then if $B'$ is a symmetric set such that for every $\gamma\in\Gamma$ we have
\[\abs{\widehat{B'}(\gamma)}\geq 2^{-1}\abs{B},\]
and furthermore $\abs{(2B'+B)\backslash B}\leq 2^{-4}\nu\alpha\abs{B}$ then
\[\norm{f\ast B'}_\infty\geq (1+2^{-3}\nu)\alpha\abs{B'}.\]
\end{lemma}
\begin{proof}
By hypothesis we have
\[\sum_{\gamma}\abs{\widehat{\mathbf{f}}(\gamma)}^2\abs{\widehat{B'}(\gamma)}^2\geq 2^{-2}\nu\alpha\norm{f}_1\abs{B'}^2N.\]
In particular, 
\[\norm{\mathbf{f}\ast B'}_2^2=N^{-1}\sum_{\gamma}\abs{\widehat{\mathbf{f}}(\gamma)}^2\abs{\widehat{B'}(\gamma)}^2\geq 2^{-2}\nu\alpha\norm{f}_1\abs{B'}^2.\]
Expanding out the $L^2$ norm we obtain that
\[\sum_x \abs{f\ast B'(x)}^2+\alpha^2\norm{B\ast B'}_2^2-2\alpha\langle B\ast B',f\ast B'\rangle\geq 2^{-2}\nu\alpha\norm{f}_1\abs{B'}^2.\]
By hypothesis
\begin{align*}
\abs{\langle B\ast B',f\ast B'\rangle-\norm{f}_1\abs{B'}^2}
&\leq \abs{B'}^2\sup_{x,y\in B'}\sum_{z\not\in B}f(z+x-y)\\
&\leq \abs{B'}^2\abs{(2B'+B)\backslash B}\\
&\leq 2^{-4}\nu\norm{f}_1\abs{B'}^2.
\end{align*}
It follows that
\[\norm{f\ast B'}_2^2\geq (1+2^{-3}\nu)\alpha\norm{f}_1\abs{B'}^2.\]
The left hand side is at most $\norm{f}_1\abs{B'}\norm{f\ast B'}_\infty$ and the lemma follows.
\end{proof}

The second lemma shows that if a set has small dimension then control on only a few elements gives control on the whole set.

\begin{lemma}\label{de2}
If $\Delta\subset\widehat{G}$ is $d$-covered by $\Gamma$ then then there is $\Lambda\subset\widehat{G}$ of size at most $d$ such that if $B$ has $(4d)^{-1}$-control of $\Lambda$ and $\Gamma(1/8)$ then for every $\gamma\in\Delta$ we have
\[\abs{\widehat{B}(\gamma)}\geq 2^{-1}\abs{B}.\]
\end{lemma}
\begin{proof}
By hypothesis there is some $\Lambda\subset\widehat{G}$ such that $\abs{\Lambda}\leq d$ and 
\[\Delta\subset \Gamma-\Gamma+ \langle \Lambda\rangle.\]
Let $\gamma\in\Delta$, so that there exists some $\gamma_0,\gamma_1\in \Gamma$ and $\epsilon\in \{-1,0,1\}^{\Lambda}$ such that
\[\gamma=\gamma_0-\gamma_1 +\sum_{\lambda\in\Lambda}\epsilon_\lambda\lambda.\]
By the triangle inequality, for every $x\in B$ we have
\[\abs{1-\gamma(x)}\leq \abs{1-\gamma_0(x)}+\abs{1-\gamma_1(x)}+d\sup_{\lambda\in\Lambda}\abs{1-\lambda(x)}\leq 1/2.\]
It follows that
\[\abs{\widehat{B}(\gamma)-\abs{B}}\leq \sum_{x\in B}\abs{1-\gamma(x)}\leq 2^{-1}\abs{B}\]
and the conclusion follows.
\end{proof}	

Finally, recalling Theorem~\ref{then} we see that to make use of Lemma~\ref{de2} we will need to have good control on the spectrum of a given set. The following lemma gives a more useful criterion to ensure this.
\begin{lemma}\label{de3}
If $c>0$ and $\abs{(B+B')\backslash B}\leq c\epsilon\abs{B}$ then $B'$ has $2c$-control of $\Delta_\epsilon(B)$.
\end{lemma}
\begin{proof}
Choose $\gamma$ such that $\abs{\widehat{B}(\gamma)}\geq \epsilon\abs{B}$ and $x\in B'$. Then
\begin{align*}
\abs{1-\gamma(x)}\abs{B}
&\leq \epsilon^{-1}\abs{\sum_{y\in B}\gamma(y)-\sum_{y\in B+x}\gamma(y)}\\
&\leq 2\epsilon^{-1}\abs{(B+B')\backslash B}\\
&\leq 2c\abs{B}.
\end{align*}
\end{proof}

We now combine all our tools thus far to prove the following efficient density increment theorem.

\begin{theorem}\label{thde}
Let $B,B'\subset G$ be any sets. Let $A\subset B$ with density $\alpha=\abs{A}/\abs{B}$ and $f:B'\to[-1,1]$ with density $\tau=\norm{f}_1/\abs{B}$. Let $\mathbf{A}(x)=A(x)-\alpha B(x)$ and suppose that
\[\sum_\gamma\abs{\widehat{f}(\gamma)}\abs{\widehat{\mathbf{A}}(\gamma)}^2=\nu\alpha\norm{f}_1\abs{A} N.\]
Then there is a finite set $\Lambda\subset\widehat{G}$ of size
\[d=\abs{\Lambda}\leq 2^{16}\mathcal{L}(\tau)(\nu\alpha)^{-1}\]
such that if a symmetric set $B''\subset G$ has $(4d)^{-1}$-control of $\Lambda$,
\[\abs{(2B''+B)\backslash B}\leq 2^{-17}\nu\alpha\abs{B},\]
and
\[\abs{(B''+B')\backslash B'}\leq 2^{-4}\exp(-2^4\mathcal{L}(\tau)\mathcal{L}(\nu\alpha))\abs{B'}\]
then there exists $x$ such that
\[\abs{(A-x)\cap B''}\geq (1+2^{-16}\nu)\alpha\abs{B''}.\]
\end{theorem}
\begin{proof}
Let $\Delta=\Delta_\eta(f)$ where $\eta=\nu\alpha/2$. In particular,
\[\sum_{\gamma\not\in\Delta} \abs{\widehat{f}(\gamma)}\abs{\widehat{\mathbf{A}}(\gamma)}^2
\leq 2^{-1}\nu\alpha\norm{f}_1\norm{\mathbf{A}}_2^2N
\leq 2^{-1}\nu\alpha \norm{f}_1\abs{A}N.\]
In particular,
\[\sum_{\gamma\in\Delta}\abs{\widehat{f}(\gamma)}\abs{\widehat{\mathbf{A}}(\gamma)}^2\geq 2^{-1}\nu\alpha\norm{f}_1\abs{A}N.\]
We now perform a dyadic decomposition of $\Delta$ into $\Delta_i=\tilde{\Delta}_{2^i\eta}(f)$, and apply Theorem~\ref{then} to each $\Delta_i$ with the weight function
\[\omega(\gamma)=\abs{\widehat{f}(\gamma)}\abs{\widehat{\mathbf{A}}(\gamma)}^2\]
and
\[\epsilon=\exp(-2^4\mathcal{L}(\nu\alpha)\mathcal{L}(\tau))\leq \exp(-8\mathcal{L}(\eta)\mathcal{L}(\tau)).\]
Then, if $\Delta_i'$ is the set provided by Theorem~\ref{then}, we have
\begin{align*}
2^{-12}\sum_{\gamma\in\Delta_i}
\abs{\widehat{f}(\gamma)}\abs{\widehat{\mathbf{A}}(\gamma)}^2
&\leq (2^i\eta)^{-1}\sum_{\gamma\in \Delta_i'}\abs{\widehat{f}(\gamma)}\abs{\widehat{\mathbf{A}}(\gamma)}^2\\
&\leq 2\norm{f}_1\sum_{\gamma\in \Delta_i'}\abs{\widehat{\mathbf{A}}(\gamma)}^2.
\end{align*}
Summing over all $i\geq 0$ implies that
\[\sum_{\gamma\in\Delta'}\abs{\widehat{\mathbf{A}}(\gamma)}^2=\sum_i \sum_{\gamma\in\Delta_i'}\abs{\widehat{\mathbf{A}}(\gamma)}^2\geq 2^{-13}\nu\alpha\abs{A}N,\]
where $\Delta'=\cup_{i\geq 0}\Delta_i'$. Furthermore, since each $\Delta_i'$ is $2^{14}\mathcal{L}(\tau)(2^i\eta)^{-1}$-covered by $\Delta_\epsilon(B')$ it follows that $\Delta'$ is $d$-covered by $\Delta_\epsilon(B')$ where
\[d\leq 2^{14}\mathcal{L}(\tau)\sum_i (2^i\eta)^{-1}\leq 2^{16}\mathcal{L}(\tau)(\nu\alpha)^{-1}.\]
Since $B''$ satisfies $\abs{(B''+B')\backslash B'}\leq 2^{-4}\epsilon\abs{B'}$ it follows from Lemma~\ref{de3} that $B''$ has $2^{-3}$-control of $\Delta_\epsilon(B')$. Hence by Lemma~\ref{de2} there is a $\Lambda$ of size $d\leq 2^{16}\mathcal{L}(\tau)(\nu\alpha)^{-1}$ such that if $B''$ has $(4d)^{-1}$-control of $\Lambda$ then for every $\gamma\in\Delta'$ we have
\[\abs{\widehat{B''}(\gamma)}\geq 2^{-1}\abs{B''}.\]
Since we also have $\abs{(2B''+B)\backslash B}\leq 2^{-17}\nu\alpha\abs{B}$ it follows from Lemma~\ref{de1} that 
\[\abs{(A-x)\cap B''}\geq (1+2^{-16}\nu)\alpha\abs{B''}\]
as required.
\end{proof}

We now at last recall our purpose, which is to study the number of solutions to \eqref{main} in a given set $A$. This count is given by $\langle (c_1\cdot A)\ast (c_2\cdot A),(-c_3\cdot A)\rangle$. It is now straightforward to use Parseval's theorem and Theorem~\ref{thde} to prove the following general density increment theorem, which shows that either this count is large or some dilation of $A$ has increased density on some structured subset.

\begin{theorem}\label{maindi}
Suppose that $B'\subset B\subset G$ and we have $A_1\subset B'$ and $A_2,A_3\subset B$ each with relative densities $\alpha_i$. Let $\alpha=2^{-1}\min(2^{-5},\alpha_1,\alpha_2,\alpha_3)$ and suppose that
\[\abs{(B'+B)\backslash B}\leq 2^{-2}\alpha\abs{B}.\]

Then either
\begin{equation}\label{eqc}\langle A_1\ast A_2,A_3\rangle \geq 2^{-2}\alpha_1\alpha_2\alpha_3\abs{B}\abs{B'}\end{equation}
or there is a finite set $\Lambda\subset\widehat{G}$ of size
\[d=\abs{\Lambda}\leq 2^{19}\mathcal{L}(\alpha)\alpha^{-1}\]
such that if a symmetric set $B''$ has $(4d)^{-1}$-control of $\Lambda$,
\[\abs{(2B''+B)\backslash B}\leq 2^{-19}\alpha\abs{B},\]
and
\[\abs{(B''+B')\backslash B'}\leq 2^{-4}\exp(-2^5\mathcal{L}(\alpha)^2))\abs{B'}\]
then there exists $x\in G$ and $i\in\{2,3\}$ such that
\[\abs{(A_i-x)\cap B''}\geq 
(1+2^{-18})\alpha_i\abs{B''}.\]
\end{theorem}
\begin{proof}
Let $\mathbf{A}_i=A_i-\alpha_iB$ for $i=2,3$ and $\mathbf{A}_1=A_1-\alpha_1B'$. We have
\[\sum_x A_1\ast\mathbf{A_2}(x)\mathbf{A}_3(x)=
\sum_x A_1\ast A_2(x)A_3(x)-\alpha_2\sum_{x\in B}A_3\ast A_1^{-}(x)\]
\[-\alpha_3\sum_{x\in B} A_1\ast A_2(x)+\alpha_2\alpha_3\sum_{x\in B}B\ast A_1(x).\]
Observe that if $A'$ is supported on $B'$ and $A$ is supported on $B$ then
\[\abs{\sum_{x\in B}A'\ast A(x)-\abs{A'}\abs{A}}\leq \abs{A'}\abs{(B'+B)\backslash B}\leq 2^{-2}\abs{A'}\abs{A}.\]
In particular, 
\[\sum_x A_1\ast\mathbf{A_2}(x)\mathbf{A}_3(x)\leq 
\sum_x A_1\ast A_2(x)A_3(x)-2^{-1}\alpha_1\alpha_2\alpha_3\abs{B}\abs{B'}.\]
By the triangle inequality either \eqref{eqc} is true or
\[\sum_{\gamma} \abs{\widehat{A_1}(\gamma)}\abs{\widehat{\mathbf{A}_2}(\gamma)}\abs{\widehat{\mathbf{A}_3}(\gamma)}\geq 2^{-2}\alpha_1\alpha_2\alpha_3\abs{B'}\abs{B}N.\]
Applying the Cauchy-Schwarz inequality it follows that 
\[\brac{\sum_\gamma \abs{\widehat{A_1}(\gamma)}\abs{\widehat{\mathbf{A}_2}(\gamma)}^2}
\brac{\sum_\gamma \abs{\widehat{A_1}(\gamma)}\abs{\widehat{\mathbf{A}_3}(\gamma)}^2}\geq 2^{-4}\abs{A_1}^2\alpha_2^2\alpha_3^2\abs{B}^2N^2.\]
It follows that for some $i\in\{2,3\}$ we have
\[\sum_\gamma \abs{\widehat{A_1}(\gamma)}\abs{\widehat{\mathbf{A}_i}(\gamma)}^2
\geq 2^{-2}\abs{A_1}\alpha_i^2\abs{B}N\]
and the conclusion then follows from Theorem~\ref{thde}. 
\end{proof}

\section{Polynomial rings}
For the first demonstration of this method we choose $G=\bbf_q[t]/(p(t))$, where $p(t)$ is some prime of degree $n$, so that $N=q^n$, and $q$ is some odd prime power. We denote this group by $\A_N$. This group is rich enough to mirror the behaviour of the integers, while it has the significant technical advantage of finite characteristic.

We fix some coefficients $c_1,c_2,c_3\in\bbf_q[t]\backslash\{0\}$ such that $c_1+c_2+c_3=0$, and some finite set $A\subset\A_N$. We suppose that $N$ is large enough so that these coefficients are coprime to $p(t)$, and hence each acts faithfully on $\A_N$. We wish to study the function that counts solutions to $c_1x_1+c_2x_2+c_3x_3=0$ with $x_i\in A$, denoted by
\begin{equation}\label{up}\Upsilon_{\mathbf{c}}(A)=\langle (c_1\cdot A)\ast (c_2\cdot A),(-c_3\cdot A)\rangle.\end{equation}
We note that $\Upsilon_{\mathbf{c}}(A)$ is invariant under dilations and translations of $A$.

We recall that both $\A_N$ and $\widehat{\A_N}$ are $\bbf_q$-vector spaces of dimension $n$. Furthermore, any $a\in\A_N$ acts on $\widehat{\A_N}$ by letting $(a\gamma)(x)=\gamma(ax)$. The Bohr space of $\Gamma\subset\widehat{\A_N}$ is defined as 
\[B(\Gamma)=\left\{ x\in\A_N : \gamma(x)=1\textrm{ for all }\gamma\in\Gamma\right\}.\]
If $B$ is a Bohr space then the frequency set of $B$ is
\[[B]=\{ \gamma\in\widehat{\A_N} : \gamma(x)=1\textrm{ for all }x\in B\},\]
and we define the rank of $B$ as the $\bbf_q$-dimension of $[B]$. 
\begin{lemma}
Every Bohr space is an $\bbf_q$-subspace of $\A_N$, and conversely, every $\bbf_q$-subspace of $\A_N$ is a Bohr space. Furthermore if $B$ is a Bohr space then 
\begin{equation}\label{brank}\abs{B}=q^{-\rk(B)}N.\end{equation}
\end{lemma}
\begin{proof}
We first claim that if $V$ is a $\bbf_q$-subspace of $\A_N$ then $\abs{V}\abs{[V]}=N$. Since $V$ is closed under addition it is easy to check that if $\gamma\not\in [V]$ then $\widehat{V}(\gamma)=0$, and if $\gamma\in[V]$ then $\widehat{V}(\gamma)=\abs{V}$, and hence by Parseval's identity, 
\[\abs{V}N=\sum_\gamma\abs{\widehat{V}(\gamma)}^2=\sum_{\gamma\in [V]}\abs{\widehat{V}(\gamma)}^2=\abs{[V]}\abs{V}^2\]
which proves our initial claim. The rest of the lemma follows easily.
\end{proof}

\begin{corollary}
If $B$ is a Bohr space and $c\in \A_N\backslash\{0\}$ then $c\cdot B$ is also a Bohr space and
\[\rk(c\cdot B)= \rk(B).\]
\end{corollary}

Since Bohr spaces are closed under addition our density increment tool, Theorem~\ref{maindi}, takes the following particularly simple form in this setting.

\begin{theorem}\label{fqtdi}
Suppose that $B\subset \A_N$ is a Bohr space and we have $A_1,A_2,A_3\subset B$ with relative densities $\alpha_i=\abs{A_i}/\abs{B}$. Let $\alpha=\min(\alpha_1,\alpha_2,\alpha_3)$. Then either
\[\langle A_1\ast A_2,A_3\rangle \gg \alpha_1\alpha_2\alpha_3\abs{B}^2\]
or there is a Bohr space $B'\subset B$ of rank
\[\rk(B')\leq \rk(B) + O(\mathcal{L}(\alpha)\alpha^{-1})\]
such that there exists $x\in \A_N$ and $i\in\{2,3\}$ such that
\[\abs{(A_i-x)\cap B'})\geq (1+\Omega(1))\alpha_i\abs{B'}.\]
\end{theorem}

We may then iterate this density increment in the usual fashion to obtain a lower bound on $\Upsilon_{\mathbf{c}}(A)$, from which Theorem~\ref{thm2} follows easily.
\begin{theorem}
Suppose that $\mathbf{c}\in(\bbf_q[t]\backslash\{0\})^3$ is such that $c_1+c_2+c_3=0$ and $N$ is sufficiently large, depending only on $\mathbf{c}$. For any $A\subset \A_N$
\[\Upsilon_\mathbf{c}(A)\geq \exp_q\brac{-O_{\mathbf{c}}(\mathcal{L}(\alpha)^2\alpha^{-1})}N^2.\]
\end{theorem}
\begin{proof}
Let $\ell=\max_{1\leq i\leq 3}\deg c_i$. Let $K$ be maximal such that there exists a sequence of Bohr spaces
\[\A_N=B_0\supset B_1\supset\cdots\supset B_K,\]
and a sequence of sets $A_i\subset B_i$ with density $\alpha_i=\abs{A_i}/\abs{B_i}$ such that, for $0\leq i< K$, there exists $\Lambda_{i+1}$ of size $O(\mathcal{L}(\alpha)\alpha_i^{-1})$ and $a_i\in \A_N\backslash\{0\}$ such that there exist $\Gamma_i$ with $B_i=B(\Gamma_i)$ and 
\begin{equation}\label{freqcontrol}\Gamma_{i+1}\subset (a_i\cdot \{1,\ldots,t^{3\ell}\}\cdot \Gamma_i)\cup \Lambda_{i+1}\end{equation}
and for some absolute constant $c>0$ we have
\[\alpha_{i+1}\geq\brac{1+c}\alpha_i,\]
and furthermore $\Upsilon_{\mathbf{c}}(A_{i+1})\leq \Upsilon_{\mathbf{c}}(A_i)$. In particular, 
\[1\geq \alpha_K \geq \brac{1+c}^K\alpha,\]
and hence $K\ll \mathcal{L}(\alpha)$. Furthermore it follows from \eqref{freqcontrol} and induction that for $1\leq J\leq K$, if $\Lambda_0=\{0\}$, there exist $a_j'\in\A_N\backslash\{0\}$ for $0\leq j<J$ such that
\[\Gamma_J\subset \Lambda_J\cup \bigcup_{j=0}^{J-1}a_j'\cdot\{1,\ldots,t^{3\ell(J-j)}\}\cdot \Lambda_j,\]
and so in particular
\begin{equation}\label{sizes}\rk(B_K)\ll_\ell K\sum_{j=0}^K\abs{\Lambda_j}\ll_\ell K\mathcal{L}(\alpha)\alpha^{-1}\sum_{i=0}^{K-1}(1+c)^{-i}
\ll_\ell \mathcal{L}(\alpha)^2\alpha^{-1}.\end{equation}
We will demonstrate that 
\begin{equation}\label{eqa}\Upsilon_{\mathbf{c}}(A_K)\gg \alpha^3\abs{B_K}^2\end{equation}
and the theorem follows from \eqref{brank} and \eqref{sizes}.

Let $B'=B(\{1,\ldots,t^{3\ell}\}\cdot \Gamma_K)$. We observe that $a\cdot B'\subset B_K$ for all $a\in\bbf_q[t]$ with $\deg a\leq 3\ell$ and if $B''$ is a Bohr space of the shape $a\cdot B'$ and $a\neq 0$ then
\[[B'']\subset a^{-1}\cdot \{1,\ldots,t^{3\ell}\}\cdot \Gamma_K.\]
Consider $B''=c_1c_2c_3\cdot B'$ and $B^{(i)}= c_i^{-1}\cdot B''$ for $1\leq i \leq 3$. Since $B^{(i)}\subset B_K$ for $1\leq i\leq 3$
\[\sum_{x\in B_K}(A_K\ast \beta^{(1)}+A_K\ast \beta^{(2)}+A_K\ast\beta^{(3)})(x)=3\abs{A_K},\]
and so for some $x\in B_K$
\[A_K\ast \beta^{(1)}(x)+A_K\ast \beta^{(2)}(x)+A_K\ast\beta^{(3)}(x)\geq 3\alpha_K.\]
If for some $1\leq i\leq 3$ we have $A_K\ast \beta^{(i)}(x)\geq \alpha_K(1+c)$ then this contradicts the maximality of $K$, letting $B_{K+1}=B^{(i)}$ and $A_{K+1}=A_K-x$. Otherwise, for $1\leq i\leq 3$, 
\[A_K\ast \beta^{(i)}(x)\geq (1-4c)\alpha_K.\]
If we let $A_i'=c_i\cdot ((A_K-x)\cap B'')$ then we have satisfied the hypotheses required to apply Theorem~\ref{fqtdi}. It follows that either \eqref{eqa} holds or we may choose $B_{K+1}$ to be the sub-Bohr set of $B_K$ with control of $\Lambda$ and $A_{K+1}$ to be a translation of some $A_{i}'$. Since 
\[\abs{\Lambda}= O(\alpha_K^{-1}\mathcal{L}(\alpha))\]
and
\[\alpha_{K+1}\geq (1+\Omega(1))(1-4c)\alpha_K\geq (1+c)\alpha_K\]
if we choose $c$ sufficiently small, this contradicts the maximality of $K$ and thus concludes the proof.
\end{proof}

\section{Integers}
We now prove Theorem~\ref{thm1} by applying these methods to the case $G=\bbz_N$, where $N$ is some large prime. Analogously to the case $G=\A_N$ we fix some coefficients $c_1,c_2,c_3\in \bbz\backslash\{0\}$, all coprime to $N$, such that $c_1+c_2+c_3=0$ and for $A\subset\bbz_N$ define $\Upsilon_{\mathbf{c}}(A)$ as in \eqref{up}. For $\Gamma\subset \widehat{\bbz_N}$ and $\rho:\Gamma\to[0,2]$ we define the Bohr set $B_\rho(\Gamma)$ as 
\[\left\{ n\in\bbz_N : \abs{\gamma(n)-1}<\rho(\gamma)\right\}.\]
We call $\Gamma$ the frequency set of $B$ and $\rho$ the width, and define the rank of $B$ to be the size of $\Gamma$. In fact, when we speak of a Bohr set we implicitly refer to the triple $(\Gamma,\rho,B_\rho(\Gamma))$, since the set $B_{\rho}(\Gamma)$ does not uniquely determine the frequency set or the width. Furthermore, if $\rho:\Gamma\to[0,2]$ and $\rho':\Gamma'\to[0,2]$ then we define $\rho\wedge \rho':\Gamma\cup\Gamma'\to[0,2]$ by
\[(\rho\wedge \rho')(\gamma)=
\begin{cases}
\rho(\gamma)&\textrm{ if }\gamma\in\Gamma\backslash\Gamma'\\
\rho'(\gamma)&\textrm{ if }\gamma\in\Gamma'\backslash\Gamma\textrm{, and }\\
\min(\rho(\gamma),\rho'(\gamma))&\textrm{ if }\gamma\in\Gamma\cap \Gamma'.\end{cases}\]

We no longer have the convenient property of Bohr sets being closed under addition, but \cite{Bo:1999} observed that certain Bohr sets have a weak version of this property suitable for our applications. A Bohr set $B_\rho(\Gamma)$ of rank $d$ is regular if for all $\abs{\kappa}\leq 2^{-6}d^{-1}$ we have 
\[\abs{\abs{B_{\rho(1+\kappa)}(\Gamma)}-\abs{B_\rho(\Gamma)}}\leq 2^6d\abs{\kappa}\abs{B_\rho(\Gamma)}.\]
If $B=B_\rho(\Gamma)$ then we write $B(\lambda)$ for the Bohr set $B_{\lambda\rho}(\Gamma)$. For further discussion of Bohr sets and proofs of the following basic lemmas see, for example, \cite[Chapter 4]{TaVu:2006}.
\begin{lemma}
For any Bohr set $B$ there exists $\lambda\in[1/2,1]$ such that $B(\lambda)$ is regular.
\end{lemma}
\begin{lemma}\label{siz}
If $\rho':\Gamma'\to[0,2]$ then
\[\abs{B_{\rho\wedge \rho'}(\Gamma\cup\Gamma')}\geq \prod_{\gamma\in\Gamma'}(\rho'(\gamma)/4)\abs{B_\rho(\Gamma)}.\]
Furthermore, if $B$ is a Bohr set of rank $d$ then $\abs{B(\lambda)}\geq \lambda^{O(d)}\abs{B}$.
\end{lemma}

The following lemma follows easily from the observation that if $c\in\bbz_N\backslash\{0\}$ then $c\cdot B_\rho(\Gamma)=B_\rho(c\cdot \Gamma)$.
\begin{lemma}
For any Bohr set $B$ and $c\in \bbz_N^*$ the set $c\cdot B$ is also a Bohr set of rank $\rk(B)$, and furthermore for any $\lambda>0$
\[c\cdot \brac{B(\lambda)}=(c\cdot B)(\lambda).\]
\end{lemma}

The notion of regularity allows us to exert the required amount of additive control, and our density increment tool, Theorem~\ref{maindi}, becomes the following.

\begin{theorem}\label{intdi}
There exists an absolute constant $c>0$ such that the following holds. Let $B\subset\bbz_N$ be a regular Bohr set of rank $d$ where $d\leq \exp(c\mathcal{L}(\alpha)^2)$. Let $A_1,A_2\subset B$ and $A_3\subset B(\delta)$, each with relative densities $\alpha_i$. Let $\alpha=\min(c,\alpha_1,\alpha_2,\alpha_3)$ and suppose that $B(\delta)$ is also regular and $cd^{-1}\alpha/4\leq \delta\leq cd^{-1}\alpha$. Then either
\begin{equation}\label{intdieq}\langle A_1\ast A_2,A_3\rangle \gg \alpha_1\alpha_2\alpha_3\abs{B}\abs{B(\delta)}\end{equation}
or there is a regular Bohr set $B'$ of rank 
\[\rk(B')\leq d+O(\mathcal{L}(\alpha)\alpha^{-1})\]
and size
\begin{equation}\label{bohrsiz}\abs{B'}\geq \exp(-O(\mathcal{L}(\alpha)^2(d+\mathcal{L}(\alpha)\alpha^{-1})))\abs{B}\end{equation}
such that there exists $x\in \bbz_N$ and $i\in\{1,2\}$ with
\[\abs{(A_i-x)\cap B'}\geq (1+c)\alpha_i\abs{B'}.\]
\end{theorem}
\begin{proof}
We first observe that since $B$ is regular we have that
\begin{equation}\label{ad0}\abs{(B(\delta)+B)\backslash B}\ll d\delta\abs{B}\leq 2^{-2}\alpha\abs{B},\end{equation}
provided $\delta$ is sufficiently small. The hypotheses of Theorem~\ref{maindi} are now met, so that either \eqref{intdieq} holds or there is a finite set $\Lambda\subset\widehat{G}$ of size $l=O(\mathcal{L}(\alpha)\alpha^{-1})$ such that if a symmetric set $B'$ has $(4l)^{-1}$-control of $\Lambda$,
\begin{equation}\label{ad1}\abs{(2B'+B)\backslash B}\leq 2^{-19}\alpha\abs{B},\end{equation}
and
\begin{equation}\label{ad2}\abs{(B'+B(\delta))\backslash B(\delta)}\leq 2^{-4}\exp(-2^5\mathcal{L}(\alpha)^2))\abs{B(\delta)}\end{equation}
then there exists $x\in \bbz_N$ and $i\in\{1,2\}$ such that
\[\abs{(A_i-x)\cap B'}\geq (1+2^{-18})\alpha_i\abs{B'}.\]
Let $\rho':\Lambda\to[0,2]$ be defined as $\rho'(\lambda)=1/4l$ for all $\lambda\in\Lambda$. If $B=B_\rho(\Gamma)$ then let $B^*=B_{\delta\rho\wedge \rho'}(\Gamma\cup\Lambda)$, and let $B'=B^*(\delta')$ for some $\delta'$ to be chosen later, but chosen such that $B'$ is regular. Clearly, $B'$ is a regular Bohr set of rank $d+O(\mathcal{L}(\alpha)\alpha^{-1})$ with $(4l)^{-1}$-control of $\Lambda$ by choice of $\rho'$. Provided $\delta'\leq 1/2$ we have that $2B'\subset B(\delta)$, and hence \eqref{ad1} is satisfied, arguing as for \eqref{ad0}. Furthermore, since $B(\delta)$ is regular and $B'\subset B(\delta\delta')$ it follows that
\[\abs{(B'+B(\delta))\backslash B(\delta)}\ll d\delta'\abs{B(\delta)},\]
and hence \eqref{ad2} is satisfied for some $\delta'\gg \exp(-O(\mathcal{L}(\alpha)^2)))d^{-1}$. Finally, by Lemma~\ref{siz} we have
\[\abs{B'}\geq (\delta')^{O(l+d)}\abs{B^*}\geq l^{-O(l)}\delta^d(\delta')^{O(l+d)}\abs{B},\]
and \eqref{bohrsiz} follows from our bounds on $\delta$, $\delta'$ and $l$, and the proof is complete.
\end{proof}

This density increment lemma may then be iterated in the standard fashion to yield a lower bound for $\Upsilon_{\mathbf{c}}(A)$. 

\begin{theorem}
Let $\mathbf{c}\in(\bbz\backslash\{0\})^3$ be such that $c_1+c_2+c_3=0$. For any prime $N$ sufficiently large, depending only on $\mathbf{c}$, and any $A\subset \bbz_N$ we have
\[\Upsilon_\mathbf{c}(A)\geq \exp\brac{-O_{\mathbf{c}}\brac{\mathcal{L}(\alpha)^4\alpha^{-1}}}N^2.\]
\end{theorem}
Theorem~\ref{thm1} follows by embedding $\{1,\ldots,N\}$ into a suitable subinterval $I$ of $\bbz_{N'}$, where $N'\ll_{\mathbf{c}}N$ is some prime large enough such that if $x_1,x_2,x_3\in I$ is a solution to $c_1x_1+c_2x_2+c_3x_3=0$ in $\bbz_{N'}$ then it is also a solution in $\{1,\ldots,N\}$. 
\begin{proof}
Let $K$ be maximal such that there exists a sequence of regular Bohr sets 
\[\bbz_N=B_0\supset B_1\supset\cdots\supset B_K,\]
each with frequency set $\Gamma_i$, rank $d_i$, width $\rho_i$ and a sequence of sets $A_i\subset B_i$ with density $\alpha_i=\abs{A_i}/\abs{B_i}$ such that, for $0\leq i<K$ we have 
\[d_{i+1}\leq d_i+O(\mathcal{L}(\alpha)\alpha_i^{-1})\leq \exp(O(\mathcal{L}(\alpha)^2)),\]
\begin{equation}\label{eqb}\abs{B_{i+1}}\geq \exp(-O(\mathcal{L}(\alpha)^2(d_i+\mathcal{L}(\alpha)\alpha_i^{-1})))\abs{B_i},\end{equation}
\[\alpha_{i+1}\geq\brac{1+c}\alpha_i,\]
and furthermore $\Upsilon_{\mathbf{c}}(A_{i+1})\leq \Upsilon_{\mathbf{c}}(A_i)$. In particular, this will follow provided $A_{i+1}$ is a subset of a dilation of a translation of $A_i$. We observe that
\[1\geq \alpha_K \geq \brac{1+c}^K\alpha,\]
and hence $K\ll\mathcal{L}(\alpha)$. Furthermore, for $0\leq j\leq K$ we have the estimate $d_j \ll  \mathcal{L}(\alpha)\alpha^{-1}$, and so
\[\abs{B_K}\geq \exp\brac{-O_{\mathbf{c}}\brac{\mathcal{L}(\alpha)^2\sum_{i=0}^Kd_i}}N\geq 
\exp(-O(\mathcal{L}(\alpha)^4\alpha^{-1}))N.\]
We will demonstrate that 
\[\Upsilon_{\mathbf{c}}(A_K)\gg \exp(-O(d_K\mathcal{L}(\alpha)^2))\abs{B_K}^2\]
and the theorem follows.

For brevity, let $B'=B_K(\rho')$, $B''=B'(\rho')=B_K(\rho'\rho'')$, and $B'''=B''(\rho''')=B_K(\rho'\rho''\rho''')$, where $\rho'$, $\rho''$ and $\rho'''$ will be chosen later, but chosen such that all Bohr sets considered are regular. Let
\[B^{(1)}=c_2c_3\cdot B''\textrm{, }B^{(2)}=c_1c_3\cdot B''\textrm{, and }B^{(3)}=c_2c_3\cdot B'''\]
so that in particular if $\rho''\ll 1/c_1c_2c_3$ we have $B^{(i)}\subset B'$ for $1\leq i\leq 3$, and hence by the regularity of $B_K$
\[\abs{\sum_{x\in B_K}A_K\ast \beta^{(i)}(x)-\abs{A_K}}\leq \abs{(B'+B_K)\backslash B_K}\ll d\rho'\abs{B_K}.\]
In particular for any $\epsilon>0$ we may choose $\rho'\gg_\epsilon d^{-1}\alpha_K$ such that
\[\sum_{x\in B_K}(A_K\ast \beta^{(1)}+A_K\ast \beta^{(2)}+A_K\ast\beta^{(3)})(x)\geq (3-\epsilon)\abs{A_K},\]
and hence, for some $x\in B$,
\[A_K\ast \beta^{(1)}(x)+A_K\ast \beta^{(2)}(x)+A_K\ast\beta^{(3)}(x)\geq (3-\epsilon)\alpha_K.\]
If $A_K\ast \beta^{(i)}(x)\geq \alpha(1+c)$ for some $1\leq i\leq 3$ then this contradicts the maximality of $K$, letting $B_{K+1}=B^{(i)}$ and $A_{K+1}=(A_K-x)\cap B_{K+1}$. In particular note that $B^{(i)}$ also has rank $d_K$ and
\[\abs{B^{(i)}}\geq \abs{B'''}\geq \exp(-O(d_K\mathcal{L}(\rho'\rho''\rho''')))\abs{B_K},\]
and we shall see that our choices for the $\rho$ parameters give the lower bound \eqref{eqb}. Otherwise, for $1\leq i\leq 3$, if we choose $\epsilon$ sufficiently small depending on $c$, then for $1\leq i\leq 3$
\[A_K\ast \beta^{(i)}(x)\geq (1-c)\alpha_K.\]
Let $A_i=c_i\cdot (A'-x)\cap c_i\cdot B^{(i)}$, and observe that
\[c_3\cdot B^{(3)}=c_1c_2c_3\cdot B'''\subset (c_1c_2c_3\cdot B'')(\rho'''),\]
and hence there exists $\rho'''\gg d_K^{-1}\alpha_K$ such that we are in a position to apply Theorem~\ref{intdi}. In particular, either
\[\Upsilon_{\mathbf{c}}(A_K)\gg \alpha_K^3\abs{B''}\abs{B'''}\gg \exp(-O(d_K\mathcal{L}(\alpha)^2))\abs{B_K}^2 \]
or there is a regular Bohr set $B^\sharp$ of rank 
\[d^\sharp\leq d + O(\mathcal{L}(\alpha)\alpha_K^{-1})\]
and size
\[\abs{B^\sharp}\geq \exp(-O(\mathcal{L}(\alpha)^2(d+\mathcal{L}(\alpha)\alpha_K^{-1})))\abs{B''}\geq \exp(-O(\mathcal{L}(\alpha)^2(d+\mathcal{L}(\alpha)\alpha_K^{-1})))\abs{B_K},\]
and $i\in\{1,2\}$ such that 
\[\abs{(A_i-x)\cap B^\sharp}\geq (1+\Omega(1))(1-c)\alpha_K\abs{B^\sharp}\geq (1+c)\alpha_K\abs{B^\sharp},\]
if we choose $c$ sufficiently small, which contradicts the maximality of $K$. This concludes the proof.
\end{proof}

\section{Arithmetic progressions in sumsets}
In this section we sketch how our improved structural result for spectra may be combined with the methods of \cite{Sa:2008} to prove Theorem~\ref{aps}. The difference lies in an improvement of the iteration lemma, Lemma 6.4 of \cite{Sa:2008}.  In particular, Theorem~\ref{aps} follows immediately from the proof of \cite{Sa:2008}, replacing the use of its Lemma 6.4 by the following quantitatively superior version.

\begin{lemma}\label{itsa}
Let $B\subset\bbz_N$ be a regular Bohr set and $A_1,A_2\subset B$ with densities $\alpha_1$ and $\alpha_2$ respectively. For any $\sigma\in(0,1]$ either
\begin{enumerate}
\item there is a regular Bohr set $B'=B(\rho)$ such that $A_1+A_2$ contains at least a proportion $1-\sigma$ of $B'$ and $\delta'\gg \alpha^2d^{-1}$, or
\item there is a regular Bohr set $B''\subset B$ of rank
\[\rk(B'')\leq \rk(B)+O(\alpha^{-1}\mathcal{L}(\sigma))\]
and width
\[\rho(B'')\gg \rho(B)(\rk(B))^{-1}\exp(-O(\mathcal{L}(\sigma)\mathcal{L}(\alpha)))\]
such that for some absolute constant $c>0$
\[\norm{A_1\ast \beta''}_\infty\norm{A_2\ast\beta''}_\infty\geq \alpha_1\alpha_2(1+c).\]
\end{enumerate}
\end{lemma}
\begin{proof}
Let $\mathbf{A}_i=A_i-\alpha B$, and $B'=B(\rho)$ for some suitable $\rho\gg \alpha^4d^{-1}$ such that $B'$ is regular. Arguing as in the proof of Lemma 6.4 in \cite{Sa:2008} we see that either the first case holds or there is some $S\subset B'$ and $1\leq i\leq 2$ such that
\[\sum_{\gamma\in\widehat{\bbz_N}}\abs{\widehat{\mathbf{A}_i}(\gamma)}^2\abs{\widehat{S}(\gamma)}
\gg \alpha_i\abs{A_i}\abs{S}N.\]
It follows from Theorem~\ref{thde} that if we choose our Bohr set $B''\subset B'$ such that it has $\Omega(l^{-1})$-control of $\Lambda$, where $\Lambda$ is a set of size $l\ll \alpha_i^{-1}\mathcal{L}(\sigma)$ and
\[\abs{(B''+B')\backslash B'}\ll \exp(-\mathcal{L}(\sigma)\mathcal{L}(\alpha_i))\abs{B'}\]
then $\norm{A_i\ast \beta''}\geq (1+c)\alpha_i$. The proof is then completed as in the proof of Lemma 6.4 in \cite{Sa:2008}.
\end{proof}

\section*{Acknowledgements}
The author would like to thank Julia Wolf and Trevor Wooley for their advice and helpful comments on a preliminary draft of this paper, and Kevin Henriot for pointing out several flaws in an earlier version of the argument.

\end{document}